\documentclass[reqno,11pt]{amsart}
\calclayout
\usepackage{amsmath,amsthm,amssymb,amsbsy}
\usepackage{amsfonts,amstext,amscd}
\usepackage{graphicx}
\usepackage{bm}
\usepackage{commath,mathtools,setspace}
\usepackage{paralist}
\usepackage{mdwlist}
\usepackage{color}
\usepackage[T1]{fontenc}
\usepackage{geometry}
\usepackage[colorlinks=true,linkcolor=blue,citecolor=blue]{hyperref}
\usepackage[colorinlistoftodos,prependcaption,textsize=small]{todonotes}
\definecolor{red}{RGB}{255,0,0}
\definecolor{green}{RGB}{0,100,0}
\definecolor{blue}{RGB}{0,0,255}
\newtheorem{theorem}{Theorem}[section]
\newtheorem{thmx}{Theorem}
\newtheorem{lemma}[theorem]{Lemma}

\newtheorem{corollary}[theorem]{Corollary}
\newtheorem{proposition}[theorem]{Proposition}

\newcommand{\sign}[1]{\mathop{\rm sign} \left(#1 \right)}
\theoremstyle{remark}
\newtheorem{remark}[theorem]{Remark}
\newtheorem{definition}[theorem]{Definition}
\newtheorem{notation}[theorem]{Notation}

\renewcommand{\Re}{\mathop{\rm Re}}

\newcommand{\D}{\mathfrak{D}}

\newcommand{\lmesh}{\mathop{\rm lmesh}}
\newcommand*\pQq[6]{%
 {\ }_{#1}{\mathcal \phi}_{#2}{\left(\genfrac..{0pt}{}{#3}{#4};#5,#6\right)}%
}
\newcommand*\lj[4]{%
 p_{#1}\left(#4;#2,#3 \right)%
}

\newcommand*\besselq[3]{%
 B_{#1}\left(#3;#2 | q\right)%
}

\renewcommand{\P}{\mathbb{P}}

\newcommand{\R}{\mathbb{R}}
\newcommand{\C}{\mathbb{C}}
\newcommand{\Z}{\mathbb{Z}}

\newcommand{\N}{\mathbb{N}}

\newcommand{\qraising}[3]{\left(#1;#2\right)_{#3}}

\newcommand{\monicpols}{\mathbb P_n^*}
\newcommand{\monicpolreal}{\mathbb P_n^*(\mathbb R)}

\newcommand{\cc}{\mathbb{C}}

\newcommand{\nn}{\mathbb{N}}

\newcommand{\rr}{\mathbb{R}}

\definecolor{apricot}{rgb}{0.98, 0.81, 0.69}

\definecolor{pastelblue}{rgb}{0.784, 0.902, 0.980}

\definecolor{greenpie}{rgb}{0.69, 0.95, 0.76}

\newcommand{\coef}[2]{e_{#1}\left(#2\right)}

\newmuskip\pFqmuskip

\newcommand*\pFqN[6][8]{%
  \begingroup 
  \pFqmuskip=#1mu\relax
  \mathcode`\,=\string"8000
  \begingroup\lccode`\~=`\,
  \lowercase{\endgroup\let~}\pFqcomma
  {}_{#2}F_{#3}{\left(\genfrac..{0pt}{}{#4}{#5};#6\right)}%
  \endgroup
}
\newcommand{\pFqcomma}{\mskip\pFqmuskip}

\title[Zeros of orthogonal little $q$-Jacobi polynomials]{Zeros of orthogonal little $q$-Jacobi polynomials: interlacing and monotonicity}

\author[A. Mart\'{\i}nez-Finkelshtein]{Andrei Mart\'{\i}nez-Finkelshtein}

\address[AMF]{Department of Mathematics, Baylor University, TX, USA, and Department of Mathematics, University of Almer\'{\i}a, Spain}

\email{a\_martinez-finkelshtein@baylor.edu}

\author[R.~Morales]{Rafael Morales}

\address[RM]{Department of Mathematics, Baylor University, TX, USA}

\email{rafael\_morales2@baylor.edu}

\author[D.~Perales]{Daniel Perales}

\address[DP]{Department of Mathematics, Texas A\&M University, TX, USA}

\email{daniel.perales@tamu.edu}

\date{\today}



\begin{document}

\begin{abstract}
We investigate the distribution of zeros of the little $q$-Jacobi polynomials and related $q$-hypergeometric families. We prove that the zeros of these orthogonal polynomials exhibit strong interlacing properties and obey natural monotonicity rules with respect to the parameters. A key tool in our approach is the logarithmic mesh, which quantifies the relative spacing of the positive real zeros and allows us to classify families of polynomials with prescribed interlacing patterns. Our results include new interlacing relations, monotonicity with respect to parameters, and structural decompositions in non-orthogonal regimes. Several classical families of $q$-hypergeometric polynomials, including $q$-Bessel and Stieltjes-Wigert polynomials, are treated as limit cases. The methods rely on a combination of classical orthogonality theory and $q$-difference equations.
\end{abstract}

\maketitle


\section{Introduction}

Orthogonal polynomials are well known for their rich zero structure, which underpins their numerous applications in approximation theory, spectral methods, and mathematical physics. The interlacing and monotonicity of zeros are especially important properties, often linked to recurrence relations and electrostatic models. In the classical continuous case, such properties are well understood, but for $q$-analogs and particularly discrete orthogonal families like the little $q$-Jacobi polynomials, many questions remain open.

In this paper, we study the real zeros of the little $q$-Jacobi polynomials
\[
p_n(x; a, b \mid q) = {}_2\phi_1\left( \begin{matrix} q^{-n}, abq^{n+1} \\ aq \end{matrix} ; q, qx \right),
\]
focusing on their interlacing patterns and dependence on the parameters $a$ and $b$. Under suitable constraints on $a$ and $b$, these polynomials are orthogonal on a discrete subset of the interval $(0,1)$, and their zeros are real, positive, and simple. We prove several new results that describe how the zeros move and interlace when the parameters vary.

A central notion in our analysis is the \textbf{logarithmic mesh}, which captures the spacing of consecutive positive zeros. This concept allows us to define classes of polynomials with controlled spacing properties and relate them to finite free convolutions. We exploit this framework to establish new interlacing relations between different families of $q$-polynomials, including $q$-Bessel and Stieltjes-Wigert polynomials, derived as suitable limits.

The approach combines tools from $q$-calculus, including $q$-difference operators and $q$-Pochhammer symbols, with intuition coming from finite free probability. We provide contiguous relations, monotonicity theorems, and structural decompositions that not only generalize earlier partial results but also correct and extend claims previously stated without proof or with incorrect arguments.

This work contributes to the understanding of $q$-orthogonal polynomials and provides a foundation for future studies involving discrete electrostatic models, asymptotic analysis, and applications in signal processing and random matrix theory.

\section{Preliminaries} \label{sec:prelim}

\subsection{Zero interlacing and mesh}\ 
 
Along this paper, $\P_n$ stands for the space of algebraic polynomials of degree $\le n$, and $\monicpols \subset \P_n$ is the subset of monic polynomials of degree $n$. Also, for a set $K\subset \C$, we denote by $\P_n(K) $ (resp., $\monicpols(K)$) the subset of polynomials of degree $\le n$ (resp., monic polynomials of degree $n$)  with all its zeros in $K$. In particular, $\monicpolreal$ denotes the family of real-rooted monic polynomials of degree $n$, 
$\monicpols(\R_{\ge 0})$ denotes the subset of $\monicpolreal$ of polynomials having only non-negative roots, etcetera. 

We denote the zeros of a polynomial $p\in \P_n$, with account of their multiplicity, by $\lambda_1(p)$, $\lambda_2(p)$, \dots, $\lambda_n(p)$. If $p\in \P_n(\R)$, we enumerate them in the increasing order,
\begin{equation}
\label{eq:zeros}
\lambda_1(p) \leq \dots \leq \lambda_n(p).    
\end{equation}

\begin{notation}[A partial order on polynomials]
\label{nota:partial.order}
Given $p,r\in\P_n(\rr)$ we say that $p \ll r$ if the zeros of $p$ are smaller than the zeros of $r$ in the following sense:
\[\lambda_k(p) \leq \lambda_k(r) \qquad \text{for all }k=1,2,\dots,n.\]
Clearly, $\ll$ is a partial order in $\P_n(\rr)$.
\end{notation}
This partial order is preserved under differentiation, see e.g.~\cite[Section 4]{arizmendi2024s}.

\begin{definition}[Interlacing] 
Let $p\in \P_n(\R)$ and $r\in \P_m(\R)$, with the zeros denoted as in \eqref{eq:zeros}.  
We say that $q$ \textbf{interlaces} $p$ (or, equivalently, that \textbf{zeros of $r$ interlace zeros of $p$}, see, e.g.,~\cite{MR3051099}), and denote it $p \preccurlyeq r$, if 
\begin{equation} \label{interlacing1}
    m=n \quad \text{and} \quad \lambda_1(p) \leq \lambda_1(r) \leq \lambda_2(p) \leq \lambda_2(r) \leq \cdots \leq  \lambda_n(p) \leq \lambda_n(r),
\end{equation}
or if
\begin{equation} \label{interlacing2}
    m=n-1 \quad \text{and} \quad \lambda_1(p) \leq \lambda_1(r) \leq \lambda_2(p) \leq \lambda_2(r) \leq \cdots \leq  \lambda_{n-1}(p) \leq \lambda_{n-1}(r) \le \lambda_n(p).
\end{equation}
Furthermore, we use the notation $p \prec r$ when all inequalities in \eqref{interlacing1} or \eqref{interlacing2} are strict. Notice that $\prec$ or $\preccurlyeq$ are not order relations, since they lack transitivity: from the fact that $p\prec r$ and $r\prec s$ it does not necessarily follow that $p\prec s$.
\end{definition}

\begin{remark}
\label{rem:interlacing.partial.order}
It follows from the definitions that if $p,r\in\P_n(\rr)$, then 
\begin{equation}
\label{eq:interlacing.implies.monotonicity}
p\prec r \qquad \Rightarrow \qquad p\ll r.
\end{equation}
The converse is not true in general, but the following partial converse will be useful: if $p,r\in\P^*_n(\rr)$ have a common interlacer, namely there exists $s\in\P_n(\rr)\setminus \P_{n-2}(\rr)$ such that $p\prec s$ and $r\prec s$, then 
$$p\ll r \qquad \Rightarrow \qquad p\prec r.$$

Indeed, the hypotheses imply that
$$\lambda_k(s)\leq \lambda_k(p)\leq \lambda_k(r)\leq \lambda_{k+1}(s) \qquad \text{for } k=1,\dots,n,$$
and the conclusion follows.
\end{remark}

The relative spacing of real zeros of a polynomial is given by the following notion: 
\begin{definition}\label{def:logarithmicmesh}
For $p\in \P_n(\R_{>0})$, $n\geq 2$, let
$$
0<\lambda_1(p) \leq \lambda_2(p) \leq \cdots \leq  \lambda_n(p)
$$
denote its zeros. The \textbf{logarithmic mesh} of $p$ is 
\begin{equation*}
    \lmesh{p}:=\max_{1\leq j\leq n-1}\frac{\lambda_{j}(p)}{\lambda_{j+1}(p)}\leq 1.
\end{equation*}
For $p\in \P_n(\R_{<0})$ we understand by $\lmesh{p}$ the value $\lmesh{r}$, with $r(x)=p(-x)$.
\end{definition}

In the rest of the paper, the polynomials with an upper bound on the logarithmic mesh will play a prominent role, which motivates the following notation: for $q\in (0,1)$ and $K$ being a subset of either $\R_{>0}$ or $\R_{<0}$, 
$$\P^q_n(K):=\{r\in \P_n(K):\lmesh{r}< q\} \quad \text{and} \quad \overline{\P^q_n}(K):=\{r\in \P_n(K):\lmesh{r}\leq q\}.
$$
We can characterize $\P_n^q(\R_{> 0})$ as follows, see \cite[Lemma 23]{lamprecht2016suffridge}:
\begin{align}
    p\in\P_n^q(\R_{> 0})& \Longleftrightarrow p(x)\prec p(qx)\label{lmeshRplus}.
\end{align}
A similar characterization is valid for $\overline{\P^q_n}(\R_{>0})$ when replacing $\prec$ by $\preccurlyeq$. Moreover, the following result is obviously true:
\begin{lemma}\label{lemma:limit}
    Let $\{p_k\}_{k=1}^\infty$ be a sequence of polynomials such that $p_k\in\P_n^*(\R)$ for sufficiently large $k\in \N$. Assume that the sequence converges point-wise to a polynomial $p\in \P^*_n(\rr)$.
    \begin{enumerate}
        \item\label{limitlmesh} if for all sufficiently large $k\in \N$, $p_k\in\P^q_n(\R)$ then $p\in\overline{\P^q_n}(\R)$.
        \item\label{limitinterlacing} if there is an $r\in \P_n(\rr)$ such that $p_k\prec r$ for all sufficiently large $k\in \N$ then $p\preccurlyeq r$.
    \end{enumerate}
\end{lemma}

\subsection{\texorpdfstring{$q$-hypergeometric polynomials}{q-hypergeometric polynomials}}\

For $a\in\cc$, $q\in(0,1)$, and $k\in \nn$, the \textbf{$q$-factorial} (also known as the \textbf{$q$-Pochhammer symbol}) is defined as
\begin{equation*}
    \qraising{a}{q}{0}:=1, \quad  \qraising{a}{q}{k}:=\prod_{j=0}^{k-1}(1-aq^{j}), \quad \qraising{a}{q}{\infty}:=\prod_{j=0}^{\infty}(1-aq^{j}).\end{equation*}
In analogy with the standard factorial and the Gamma function, we can extend this notion to the complex plane: for $\lambda\in \C$,
\begin{equation}\label{continuousqraising}
    \qraising{a}{q}{\lambda} :=\frac{\qraising{a}{q}{\infty}}{\qraising{aq^{\lambda}}{q}{\infty}}=\prod_{j=0}^\infty\frac{1-aq^j}{1-aq^{\lambda+j}},
\end{equation}
where we use the branch of $q^\lambda$ such that $|q^\lambda|=q^{\Re (\lambda)}$. We can extend the definition to vector parameters $\bm a$: if $\bm a=(a_1, \dots,a_r)\in\R^r$, then 
\begin{equation*}
    \qraising{\bm a}{q}{k}:=\prod_{j=1}^r \qraising{a_j}{q}{k}.
\end{equation*}

With this notation, the \textbf{$q$-hypergeometric series} for $\bm a=(a_1,\cdots,a_r)\in\R^r$ and $\bm b=(b_1,\cdots,b_s)\in\R^s$ is (see \cite[\S 1.10]{koekoek2010hypergeometric})
\begin{equation*}
    \pQq{r+1}{s}{a_0,\bm a}{\bm b}{q}{x}:=\sum_{k=0}^{\infty}\frac{\qraising{a_0}{q}{k}  \qraising{\bm a}{q}{k}}{\qraising{\bm b}{q}{k}}(-1)^{(s-r)k}q^{(s-r)\binom{k}{2}}\frac{x^k}{\qraising{q}{q}{k}},
\end{equation*}
where we assume $\binom{0}{2}=\binom{1}{2}=0$. 
When for some $j$, $a_j=q^{-n}$ with $n\in\N$, the series truncates: 
\begin{align}\label{def:qpolynomial1}
     \pQq{r+1}{s}{q^{-n},\bm a}{\bm b}{q}{x}:=\sum_{k=0}^{n}\frac{\qraising{q^{-n}}{q}{k}  \qraising{\bm a}{q}{k}}{\qraising{\bm b}{q}{k}}(-1)^{(s-r)k}q^{(s-r)\binom{k}{2}}\frac{x^k}{\qraising{q}{q}{k}}
\end{align}
is a $q$-\textbf{hypergeometric polynomial} of degree $\le n$, under the assumption that
\begin{equation}  \label{constrainB}
    b_1,\dots b_s \in \C \setminus \{ q^{-1},q^{-2},\dots, q^{-n+1} , q^{-n}\}.
\end{equation}

There are some notorious classical families of polynomials that are $q$-hypergeometric.

The \textbf{$q$-Laguerre} polynomial of degree $n$ and parameter $b\in\R$, denoted by $L^{(b)}_n$ (where we use a slightly different notation from the one in \cite{koekoek2010hypergeometric}), is
\begin{equation}\label{QHGFLaguerre}
      L_n^{(b)}(x;q) = \frac{\qraising{b}{q}{n}}{\qraising{q}{q}{n}} \pQq{1}{1}{q^{-n}}{b}{q}{-q^{n}bx}.
\end{equation}
For $b\in (0,1)$, it satisfies two kinds of orthogonality relations on $(0,\infty)$, an absolutely continuous one and a discrete one. The discrete weight is supported on a (possibly, rescaled) lattice $\{q^k:\, k\in \Z\}$; see \cite[Section 14.21]{koekoek2010hypergeometric} for details. In this setting, the zeros are real and positive, each interval $(q^{k+1},q^k)$, $k\in \Z$, contains at most one zero, and (see \cite[Theorem 4]{moak1981q}),  
\begin{equation}\label{qlaguerrelmesh}
    L_n^{(b)}(x;q)\in \P^{q^2}_n(\R_{>0})\subset\P^{q}_n(\R_{> 0}).
\end{equation}

Further results on the roots of the $q$-Laguerre polynomials for $b\in(0,1)$ can be found in \cite{moak1981q, tcheutia2018mixed}. For instance, they decrease with $b$, and (by \cite[Theorem 23]{tcheutia2018mixed}),  
$$
L_n^{(b)}(x;q)\prec L_n^{(q^2b)}(x;q).
$$
Combining these facts, we conclude that 
for $q^2<t<1$,
\begin{equation}\label{qLaguerreinterlacing}
    L_n^{(b)}(x;q)\prec L_n^{(tb)}(x;q).
\end{equation}

In the  limit $b\to 0$, the $q$-Laguerre polynomials become the \textbf{Stieltjes-Wigert} polynomials,
\begin{equation}\label{Stieltjes-Wigert}
    S(x):=\lim_{b\to 0}L_n^{(b)}(qb^{-1}x;q)=\frac{1}{\qraising{q}{q}{n}}\pQq{1}{1}{q^{-n}}{0}{q}{-q^{n+1}x}.
\end{equation}
By Lemma \ref{lemma:limit},  
\begin{equation}\label{Stieltjeslmesh}
S(x)\in\overline{\P_n^{q^2}}(\R_{>0}).
\end{equation}

Another relevant family of classical $q$-hypergeometric polynomials, depending on two parameters, $a, b\in \R$, are the \textbf{little $q$-Jacobi polynomials}: 
\begin{equation}\label{def:LittleqJacobi}
   p_n(x;a,b|q) := \pQq{2}{1}{q^{-n}, abq^{n+1}}{aq}{q}{qx}.
\end{equation}
These polynomials are the primary focus of this paper. To simplify the notation, from now on we write $p_n(x;a,b)$ instead of $p_n(x;a,b|q)$, whenever it cannot cause confusion.

Under the assumption that $0<qa<1$ and $qb<1$, these polynomials are orthogonal on $(0,1)$  with respect to the discrete measure
\begin{equation}\label{LittleJacobiContinuousweight}
    \sum_{k=0}^\infty\frac{(bq;q)_k}{(q;q)_k} (aq)^k\delta_{ q^k}=\left( \frac{\qraising{bq}{q}{\infty}\qraising{xq}{q}{\infty}}{\qraising{bxq}{q}{\infty}\qraising{q}{q}{\infty}}\, x\, a^{\frac{\ln{x}}{\ln{q}}}\right)d\lambda, \quad d\lambda:= \sum_{k=0}^\infty \delta_{ q^k},
\end{equation}
where $\delta_{x}$ is the Dirac delta (unit mass point) at $x$.  Again, from the general properties of orthogonal polynomials it follows that $\lj{n}{a}{b}{x}\in\P_n((0,1))$, and each interval $(q^{k},q^{k-1})$, $k\in \N$, contains at most one zero. 

A particular case of the little $q$-Jacobi family are the \textbf{little $q$-Laguerre polynomials} \cite{koekoek2010hypergeometric}, that are obtained by setting $b=0$ in \eqref{def:LittleqJacobi}:
\begin{equation*}
    p_n(x;a|q):= p_n(x;a,0|q)=\pQq{2}{1}{q^{-n}, 0}{aq}{q}{qx}.
\end{equation*}
Finally, in the  limit $a\to 0$, with $b<0$, the little $q$-Jacobi polynomials become the \textbf{$q$-Bessel} polynomials,
\begin{equation}\label{def:qbessel}
    \besselq{n}{b}{x}:=\lim_{a\to 0} p_n(x;a,\tfrac{b}{qa}\left. \right|q)=\pQq{2}{1}{q^{-n},bq^n}{0}{q}{x}.
\end{equation}

\subsection{\texorpdfstring{The $q$-derivative and the logarithmic mesh}{The q-derivative and the logarithmic mesh}}\

A technical tool that is useful in studying the relative spacing of the zeros of a polynomial is a $q$-analogue of the derivative. For $q\in (0,1)$, the \textbf{$q$-derivative} operator $\D_q$ is defined as
\begin{equation}\label{def:qderivative}
    \D_q f(x) := \begin{cases} 
\dfrac{f(x) - f(qx)}{(1-q)x}, & x \neq 0, \\ 
f'(0), & x = 0.
\end{cases}
\end{equation}
Some elementary properties are  
\begin{equation}\label{qderivative1}
  \D_q c=0,\quad   \D_q[f(cx)]=c(\D_qf)(cx), \quad c\in\R,
\end{equation}
and
$$
\D_q x^i = \frac{1-q^i}{1-q}\, x^{i-1},\quad i\in\N.
$$
In particular, for $p\in\P_n$ written in the form 
\begin{equation}\label{basicform}
    p(x)=\sum_{i=0}^n  e_i(p) \, x^{i} ,
\end{equation}
where $  e_i(p) $ are the coefficients of $p$ in the monomial basis (a notational convention we will use in the rest of the paper), we get 
\begin{equation*}
    \D_q p(x)=\sum_{i=1}^n\frac{1-q^i}{1-q} e_i(p)x^{i-1}=\sum_{i=0}^{n-1}\frac{1-q^{i+1}}{1-q}e_{i+1}(p)x^{i}.
\end{equation*}
In our notation,
\begin{equation}\label{qderivativecoefficient}
    e_i(\D_qp)=\frac{1-q^{i+1}}{1-q}e_{i+1}(p), \quad i=0, 1, \dots, n-1.
\end{equation}
In particular, straightforward calculations allow us to verify the identity
\begin{equation}
    \label{corollary:qderivativeqhyper}
    \D_q  \pQq{r+1}{s}{q^{-n},\bm a}{\bm b}{q}{x} = c\pQq{r+1}{s}{q^{-n+1}, q\bm a}{q\bm b}{q}{q^{s-r}x},
\end{equation}
where 
\begin{equation}
    \label{corollary:qderivativeqhyper2}
c=(-1)^{s-r}\frac{(1-q^{-n})\prod_{j=1}^r(1-a_j)}{(1-q)\prod_{j=1}^s(1-b_j)} .
\end{equation}

The following result is a well-known fact, see for instance \cite[theorem 25]{lamprecht2016suffridge}.

\begin{proposition}\label{prop:lmeshinterlaing}
If $p \in \P^q_n(\rr_{> 0})$, then  $\D_q p \in \P^q_{n-1}(\rr_{> 0})$ and $p\prec \D_q p$.
\end{proposition}

\section{\texorpdfstring{Zeros of Some Classical $q$-Polynomials}{Zeros of Some Classical q-Polynomials}} \label{sectionlittleqJacobi}

\subsection{\texorpdfstring{Zeros of little $q$-Jacobi and $q$-Bessel polynomials}{Zeros of little q-Jacobi and q-Bessel polynomials}}\

Recall our convention that we consider $q\in (0,1)$ fixed, and write $p_n(x;a,b)$ instead of $p_n(x;a,b|q)$.

Compared to the extensive results available for the zeros of $q$-Laguerre polynomials, significantly less is known about the zeros of the little $q$-Jacobi polynomials. 
The following interlacing properties, corresponding to the case when these are orthogonal with respect to the discrete measure \eqref{LittleJacobiContinuousweight}, were known:
\begin{thmx}\label{thmx:interlacings}
Let $0<aq<1$ and $bq<1$, the following interlacings hold:
\begin{align}
 &\lj{n+1}{a}{b}{x}\prec\lj{n}{a}{qb}{x},  \label{eq:interlacing1} \\
 & \lj{n+1}{a}{b}{x}\prec\lj{n}{a}{q^2b}{x}, \qquad \text{and} \label{eq:interlacing2} \\
 & \lj{n}{a}{q^2b}{x}\prec \lj{n}{qa}{qb}{x}.  \label{eq:interlacing3}
\end{align}
\end{thmx}
Theorem~\ref{thmx:interlacings} was proved in \cite[Theorem 12 d), e)]{tcheutia2018mixed} for $b<0$, but the general statement follows using the same arguments. It is important to observe that the interlacing 
$$
\lj{n+1}{a}{b}{x} \prec \lj{n}{qa}{qb}{x},
$$
established in Theorem~\ref{Thm:propertieslittleqjacobi} \textit{(ii)} below, is not a consequence of Theorem~\ref{thmx:interlacings} due to lack of transitivity in $\prec$. 

The following theorems contain new results under the same assumptions:
\begin{theorem}[Monotonicity of zeros]\label{Thm:propertieslittleqjacobi}
    If $0<aq<1$ and $bq<1$ then the zeros of the little $q$-Jacobi polynomials $p_{n}(x;a,b)$ are increasing with respect to $b$ and decreasing with respect to $a$. 
\end{theorem}
In terms of the partial order from Notation \ref{nota:partial.order}, Theorem \ref{Thm:propertieslittleqjacobi} implies that for $0<a_1<a_2<\tfrac{1}{q}$ and $b_1<b_2<\tfrac{1}{q}$ we have
    $$p_{n}(x;a_2,b_1)\ll p_{n}(x;a_1,b_1)\ll p_{n}(x;a_1,b_2).$$
In analogy with the standard Jacobi case, this suggests the existence of an electrostatic model for the zeros of the little $q$-Jacobi polynomials. We are not aware of such results in the literature. 

\begin{theorem}[lmesh and interlacing of zeros]\label{Thm:propertieslittleqjacobi2}
    If $0<aq<1$ and $bq<1$, then  
    $$\lj{n}{a}{b}{x}\in\P_n^q((0,1)),$$ and the following interlacing holds:
        \begin{enumerate}
        \item[(i)]\label{thm:interlacing2} $\lj{n}{a}{b}{x}\prec\lj{n-1}{qa}{qb}{x}$,
        \item[(ii)] \label{thm:interlacing1}$\lj{n}{a}{q^2b}{x}\prec\lj{n}{q^2a}{b}{x}$.
    \suspend{enumerate}
Additionally, for $b<0$,
    \resume{enumerate}
    \item[(iii)]\label{thm:interlacing3} $\lj{n}{a}{b}{x}\prec\lj{n}{a}{q^2b}{x}$.
    \end{enumerate} 
\end{theorem}

A useful consequence of the previous result is the following:
\begin{corollary}\label{cor:interlacinglittleqJacobi} 
    Let $0<aq<1$ and $q^2\leq t_1,t_2\leq1$ such that $t_1t_2\not=1$. Then 
     \begin{enumerate}
    \item[(i)]\label{thm:continuousinterlacing1} for $0\leq qb<1$,
         \begin{equation*}
              \lj{n}{a}{t_1b}{x}\prec\lj{n}{t_2a}{b}{x};
         \end{equation*}
\item[(ii)]\label{thm:continuousinterlacing2} for $b<0$,
         \begin{equation*}
         \lj{n}{a}{b}{x}\prec\lj{n}{a}{t_1b}{x}.
     \end{equation*} 
     \end{enumerate}     
\end{corollary}
Part of the statement of Corollary~\ref{cor:interlacinglittleqJacobi} was announced in \cite[Theorem 8]{gochhayat2016interlacing}, but the proof contained errors. This result generalizes several interlacing properties appearing in \cite{gochhayat2016interlacing, tcheutia2018mixed}. 

If $b=q^{-k}$ for $k\in \{1,\cdots,n\}$, when the little $q$-Jacobi polynomials are no longer orthogonal on the real line, we have:
\begin{proposition}
 \label{littleqjacobimultiplicity}  
 For $k\in\{1,\dots,n\}$, let 
\begin{equation}\label{def:identitymultiplicaitve}
    E_{k}(x):=\prod_{j=1}^k \left( 1-q^{-j}x \right)=\pQq{1}{0}{q^{-k}}{\cdot}{q}{x}.
\end{equation} 
Then
\begin{align} \label{factorization}
    \lj{n}{a}{q^{-k}}{x}=E_{k}(qx)\, \lj{n-k}{a}{q^k}{q^{-k}x},
\end{align}
and $\lj{n}{a}{q^{-k}}{x}\in\overline{\P_n^q}((0,1))$. 
\end{proposition}

The case $a=q^{-k}$, with $k\in\{1,\dots,n\}$, violates the constrains \eqref{constrainB}. However, we can overcome this difficulty by considering the normalized little $q$-Jacobi polynomials
\begin{equation}\label{normalizedqJacobi}
        \qraising{q^{-k+1}}{q}{n}\lj{n}{q^{-k}}{b}{x}:=\sum_{j=0}^n\frac{\qraising{q^{-n}}{q}{j}}{\qraising{q}{q}{j}}\qraising{bq^{n-k+1}}{q}{j}\qraising{q^{j-k+1}}{q}{n-j}(qx)^j.
\end{equation}
\begin{proposition}\label{littleqjacobimultiplicity0} 
    For $k\in\{1,\dots,n\}$, 
    \begin{equation*}
        \qraising{q^{-k+1}}{q}{n}\lj{n}{q^{-k}}{b}{x}=c\, (qx)^k\, \lj{n-k}{q^k}{b}{x},
    \end{equation*}
    where
    \begin{equation*}
        c=(-1)^kq^{\binom{k}{2}-nk}\qraising{bq^{n-k+1}}{q}{k}\qraising{q^{k+1}}{q}{n-k}.
    \end{equation*}
\end{proposition}

Observe that in this case, the polynomials present a multiple zero at the origin; the definition of the logarithmic mesh is not applicable and would need to be modified to formulate an assertion analogous to that in 
Proposition~\ref{littleqjacobimultiplicity}.

\medskip 

Since $q$-Bessel polynomial $\besselq{n}{b}{x}$ is a limit case of the little $q$-Jacobi polynomial (see~\eqref{def:qbessel}), some properties of the latter are still valid. For instance, $\besselq{n}{b}{x}\in \overline{\P_n^q}((0,1))$. But more is true:
\begin{theorem}
    \label{thm:lmeshqbessel}
    Let $b<0$. Then, 
    \begin{equation}\label{eq:Besselrealzeros}
        \besselq{n}{b}{x}\in\P_n^q((0,1)).
    \end{equation} 
    If additionally $q^2<t<1$, then
     \begin{equation}\label{eq:Besselinterl}
         \besselq{n}{b}{x}\preccurlyeq \besselq{n}{tb}{x}.
     \end{equation}
\end{theorem}

\subsection{\texorpdfstring{Zeros of $_1\phi_1$, $_2\phi_0$, and $_2\phi_1$ $q$-hypergeometric polynomials}{Zeros of 1phi1, 2phi0, and 2phi q-hypergeometric polynomials}}\

We can use the expression of $q$-Laguerre, little $q$-Jacobi, little $q$-Laguerre and $q$-Bessel polynomials as $q$-hypergeometric polynomials (see \eqref{QHGFLaguerre}, \eqref{Stieltjes-Wigert}, \eqref{def:LittleqJacobi}, and \eqref{def:qbessel}) to recast the results above in terms of the zeros of some families of $_2\phi_1$ and $_1\phi_1$ polynomials. We can additionally use the identities  \cite[(1.13.14), (1.13.15), (1.13.16)]{koekoek2010hypergeometric}, 
namely, 
\begin{align}
    \pQq{1}{1}{q^{-n}}{b}{q}{x}&=\frac{x^n}{\qraising{b}{q}{n}q^n}\pQq{2}{1}{q^{-n},b^{-1}q^{1-n}}{0}{q}{\frac{bq^{n+1}}{x}}\label{qlaguerrereciprocal}, \\
    \pQq{2}{1}{q^{-n},a}{b}{q}{z} &=
    \frac{\qraising{a}{q}{n}}{\qraising{b}{q}{n}} q^{-n-\binom{n}{2}}(-z)^n
    \pQq{2}{1}{q^{-n},b^{-1}q^{1-n}}{a^{-1}q^{1-n}}{q}{\frac{b q^{n+1}}{az}}\label{littleqjacobireciprocal1}, \\
    \pQq{2}{0}{q^{-n},a}{-}{q}{x}&=\frac{x^n}{q^n}\qraising{a}{q}{n}\pQq{2}{1}{q^{-n},0}{a^{-1}q^{1-n}}{q}{\frac{q^{n+1}}{ax}},
    \label{littleqjacobireciprocal2}
\end{align}
to extend it to zeros of some additional families of $_2\phi_0$, $_1\phi_1$, and $_2\phi_1$ polynomials.

\subsubsection{Zeros of a family of $_1\phi_1$ polynomials} 
Polynomials
$$ 
\pQq{1}{1}{q^{-n}}{b}{q}{x}
$$
are related, by \eqref{QHGFLaguerre} and \eqref{Stieltjes-Wigert}, to $q$-Laguerre polynomials and the Stieltjes-Wigert polynomials. Also, by \eqref{qlaguerrereciprocal}, they can be rewritten in terms of $_2\phi_1$ polynomials, and thus, by \eqref{def:qbessel}, of some $q$-Bessel polynomials. In particular, \eqref{qlaguerrelmesh}, \eqref{Stieltjeslmesh}, and \eqref{eq:Besselrealzeros} yield the locations of the zeros and an upper bound for the logarithmic mesh of these polynomials, 
\begin{equation}
    \pQq{1}{1}{q^{-n}}{b}{q}{x}\in\begin{cases}
        \P^{q^2}_n(\R_{<0})\subset\P^{q}_n(\R_{< 0}) \quad &\text{ for } b>0, \\        \overline{\P^{q^2}_n}(\R_{<0})\subset\P^{q}_n(\R_{< 0})\quad & \text{ for } b=0, \\
        \P^{q}_n(\R_{< 0}) \quad &\text{ for } b<0,
    \end{cases}
\end{equation}
while by \eqref{qLaguerreinterlacing} and \eqref{eq:Besselinterl} we have the interlacing properties: for  $q^2\leq t<1$,
\begin{align}
    \pQq{1}{1}{q^{-n}}{tb}{q}{x}&\preccurlyeq\pQq{1}{1}{q^{-n}}{b}{q}{x} \quad \text{for } b\geq 0 ,\\
     \pQq{1}{1}{q^{-n}}{b}{q}{x}&\preccurlyeq \pQq{1}{1}{q^{-n}}{t^{-1}b}{q}{t^{-1}x} \quad \text{for } b< 0.
\end{align}

\subsubsection{Zeros of a family of $_2\phi_1$ polynomials} 

The relation of the little $q$-Jacobi \eqref{def:LittleqJacobi} and $q$-Bessel polynomials \eqref{def:qbessel} with the $_2\phi_1$ functions, along with Theorems \ref{Thm:propertieslittleqjacobi} and \ref{thm:lmeshqbessel}, yield results on zeros of the family 
\begin{equation}
    \pQq{2}{1}{q^{-n},a}{b}{q}{x}
\end{equation}
for a range of parameters $a$ and $b$. Some of them are summarized in  Table \ref{tab:2Q1}, see rows 1, 2 and 5. Additional information about the zeros of $_2\phi_1$ polynomials follows from using the reciprocals of $q$-Laguerre polynomials \eqref{qlaguerrereciprocal} (see row 6 of Table \ref{tab:2Q1}). Finally,  \eqref{littleqjacobireciprocal1} connects two $_2\phi_1$ polynomials; using this connection,  and row 1 of Table \ref{tab:2Q1}  we obtain the rows 3 and 4 of Table \ref{tab:2Q1}.

Monotonicity results follow from Corollary \ref{cor:interlacinglittleqJacobi}: for $q^2\leq t_1,t_2\leq1$ such that $t_1t_2\not=1$, and for $0<b<1$, if $0\leq a< bq^{n-1}$, 
\begin{equation}\label{2phi1interlacing1}
 \pQq{2}{1}{q^{-n},t_1a}{b}{q}{x}\prec\pQq{2}{1}{q^{-n},t_2a}{t_2b}{q}{x} ,
\end{equation}
while if $a<0$,
\begin{equation}\label{2phi1interlacing2}
 \pQq{2}{1}{q^{-n},a}{b}{q}{x}\prec\pQq{2}{1}{q^{-n},t_1a}{b}{q}{x} .
\end{equation}

\subsubsection{Zeros of a family of $_2\phi_0$ polynomials.} 

Identity \eqref{littleqjacobireciprocal2} gives us a connection between $_2\phi_0$ and $_2\phi_1$ polynomials, namely,  
$$
\pQq{2}{0}{q^{-n},a}{-}{q}{x}=\frac{x^n}{q^n}\qraising{a}{q}{n}\pQq{2}{1}{q^{-n},0}{a^{-1}q^{1-n}}{q}{\frac{q^{n+1}}{ax}},
$$
so we can use Row 1 of Table \ref{tab:2Q1}, under the assumption that $a>q^{1-n}$, to get the result in Row 7 of Table \ref{tab:2Q1}. Finally, by Corollary \ref{cor:interlacinglittleqJacobi},
we conclude that for $q^2<t<1$,
\begin{equation*}
    \pQq{2}{0}{q^{-n},t^{-1}a}{-}{q}{tx} \prec \pQq{2}{0}{q^{-n},a}{-}{q}{x}.
\end{equation*}

\begin{table}[h]
    \centering
\begin{tabular}{|c|c|c|c|}
\hline $a$ & $b$ & \textbf{Roots in} & $\lmesh$ \\
\hline \hline $\left(-\infty, bq^{n-1}\right)$ & $(0,1)$ & $(0,q)$ & $<q$ \\[1mm]
\hline $\Big\{b,bq,\dots, bq^{n-1} \Big\} $ & $(0,1)$ & $(0,q)$ & $\leq q$ \\
\hline $\Big(q^{1-n},\infty\Big)$ & $(-\infty,0)$ & $\R_{<0}$ & $<q$ \\
\hline $\Big(q^{1-n},bq^{n+1}\Big)$ & $\Big(q^{-2n+2},\infty\Big)$ &  $\R_{>0}$  & $<q$ \\
\hline $(-\infty,0)$ & $\{0\}$ & $(0,q)$ & $<q$\\
\hline  $\Big(q^{1-n},\infty\Big)$  & $\{0\}$ & $\R_{<0}$  & $<q$ \\
\hline $\Big(q^{1-n},\infty\Big)$ & $\cdot$ & {$\R_{> 0}$} & $<q$ \\
\hline  $\cdot$ & $(0,1)$ & $\R_{<0}$ & $<q^2$ \\
\hline  $\cdot$ & $0$ & $\R_{<0}$ & $\leq q^2$ \\
\hline $\cdot$ & $(-\infty,0)$ & $\R_{<0}$ & $<q$ \\
\hline
\end{tabular}
\vspace{2mm}
\caption{Real zeros of some $\pQq{i+1}{j}{q^{-n}, \, a}{  b}{q}{x}$ polynomials, $i,j\in\{0,1\}$.
}
    \label{tab:2Q1}
\end{table}   

\section{Proofs} \label{sectionProofs}

\subsection{Proof of Theorem \ref{Thm:propertieslittleqjacobi}}

The proof is a direct application of the following result; see, e.g. \cite[Lemma 2]{gochhayat2016interlacing}:
let $P_n(x, \tau)$ be orthogonal with respect to $\omega(x; \tau) d\lambda(x)$, where $\omega(x; \tau)$ is positive and continuous functions of $x$ on an interval $[\alpha, \beta]$ for every $\tau\in (c,d)$. If the ratio $\frac{\omega(x; \tau_2)}{\omega(x; \tau_1)}$ is an increasing (decreasing) function of $x$ for each $\tau_1,\tau_2\in(c,d)$ such that $\tau_1<\tau_2$, then the zeros of $P_n(x, \tau)$ are increasing (decreasing) with respect to $\tau$.

By \eqref{LittleJacobiContinuousweight} we have that $p_{n}(x;a,b|q)$ is orthogonal with respect to $\omega(x;a, b)d\lambda(x) $, where
    $$
    \omega(x;a,b)=\frac{\qraising{bq}{q}{\infty}\qraising{xq}{q}{\infty}}{\qraising{bxq}{q}{\infty}\qraising{q}{q}{\infty}}xa^{\frac{\ln{x}}{\ln{q}}}.
    $$
 For parameters $b_1<b_2<\frac{1}{q}$, since
    $$
    \frac{\omega(x;a,b_2)}{\omega(x;a,b_1)}=\frac{\qraising{b_2q}{q}{\infty}\qraising{b_1xq}{q}{\infty}}{\qraising{b_2xq}{q}{\infty}\qraising{b_1q}{q}{\infty}}=\prod_{k=0}^\infty\frac{(1-b_2q^{k+1})(1-b_1xq^{k+1})}{(1-b_1q^{k+1})(1-b_2xq^{k+1})},
    $$
    we get that
    \begin{align*}
      \frac{d}{dx}\left(\frac{\omega(x;a,b_2)}{\omega(x;a,b_1)}\right)&=\frac{\omega(x;a,b_1)}{\omega(x;a,b_2)}\, \sum_{k=0}^\infty q^{k+1}\left(\frac{b_2}{1-b_2xq^{k+1}}-\frac{b_1}{1-b_1xq^{k+1}}\right)\\ &=\frac{\omega(x;a,b_1)}{\omega(x;a,b_2)}\, \sum_{k=0}^\infty q^{k+1}\frac{b_2-b_1}{(1-b_1xq^{k+1})(1-b_2xq^{k+1})}.
    \end{align*}
   Thus,
    $$
    \sign{\frac{d}{dx}\left(\frac{\omega(x;a,b_2)}{\omega(x;a,b_1)}\right)}=\sign{b_2-b_1}>0.
    $$
     It implies that the zeros of $\lj{n}{a}{b}{x}$ increase with respect to $b$. 
     
     Similarly, for the parameter $a$, let $a_1<a_2$ such that $0<a_1<a_2<\frac{1}{q}$. Differentiation with respect to $x$ of the ratio
    $$
    \frac{\omega(x;a_2,b)}{\omega(x;a_1,b)}=\left(\frac{a_2}{a_1} \right) ^{\frac{\ln x}{\ln q}} 
    $$
    yields
    $$
    \frac{d}{dx}\left(\frac{\omega(x;a_2,b)}{\omega(x;a_1,b)}\right)=\left(\frac{a_2}{a_1} \right) ^{\frac{\ln x}{\ln q}}\frac{\ln\left(\frac{a_2}{a_1}\right)}{x\ln q} <0, \qquad \text{for } x\in(0,1).
    $$

\subsection{Proof of Theorem \ref{Thm:propertieslittleqjacobi2}}

Our approach to prove the interlacing results is standard: if $g,p,r$ satisfy a contiguous relation and $p$ interlaces $r$, then $g$ also interlaces $p$ and $r$. Specifically, we will use the following result:
\begin{lemma}[\cite{jordaan2009interlacing}]\label{Jordaaninterlacing}
     Let $p_n$ and $r_{n-1}$ be polynomials of degree $n$ and $n-1$, respectively, with their zeros in a finite interval $(a,b)$, and such that $p_n\prec r_{n-1}$. For polynomials $A$ and $B$ that have a constant sign on $(a,b)$, consider 
    $$
    g(x)=A(x)p_n(x)+B(x)r_{n-1}(x).
    $$
Then
\begin{itemize}
    \item if the degree of g is $n$, then $g\in \P_n(\R)$, $g \prec r_{n-1}$, and either 
    $$
    p_n\prec g   \quad \text{ or } \quad  g\prec p_n ;
    $$
    \item if the degree of g is $n-1$, then $g\in \P_{n-1}(\R)$, $p_n\prec g$, and either
    $$
  g \prec r_{n-1} \quad \text{ or } \quad   r_{n-1} \prec g.
    $$
\end{itemize}
\end{lemma}

We will also need the following contiguous relations.
\begin{proposition}Let $n \in \mathbb{N}$ and $b \neq q^{-m}$ for $m = 0, 1, 2, \ldots$. Then the following relations hold:
    \begin{align} \label{contiguousrelation1}
        -\lj{n}{a}{q^2b}{x}& = \tfrac{a(1-q^n)(1-abq^{n+3})}{q^{n-2}(1-aq)(1-aq^2)}x\lj{n-1}{q^2a}{q^2b}{x}-\lj{n}{qa}{qb}{x}.
    \end{align}
   \begin{equation}\label{contiguousrelation2}
   \begin{split}
        (1-aq)&(1+bq^{n}(aq^{n+1}-aq-1))\lj{n}{a}{b}{x}=\\&(1-bq^{n})(1-aq^{n+1})\lj{n}{qa}{\tfrac{b}{q}}{x}\\
        & +aq(1-q^n)(1-abq^{n+1})(bqx-1)\lj{n-1}{qa}{qb}{x}.
   \end{split}
        \end{equation}
 \begin{align}\label{contiguousrelation3}
        q^n(1-abq^n)\lj{n}{a}{b}{x}=(1-abq^{2n})\lj{n}{a}{\frac{b}{q}}{qx}-(1-q^n)\lj{n-1}{a}{b}{x}.
    \end{align}  
    \begin{align} \label{contiguousrelation4}
     bq^{n+1} \left(1-aq^{n+1} \right) &\lj{n+1}{a}{b}{x}=
     \\&(1 - ab q^{2n+2}) \left(1 - qbx\right) \lj{n}{a}{qb}{x} -\left(1-b q^{n+1}\right) \lj{n}{a}{b}{x}.\nonumber
    \end{align}
  
\end{proposition}
The recurrence relation \eqref{contiguousrelation4} has appeared previously in  \cite[Eq. (7)]{gochhayat2016interlacing}, and we include it here for completeness. 
\begin{proof}
We can prove these identities by showing that for every $k$, the $k$-th coefficients in both sides match.

    For instance, to show \eqref{contiguousrelation1} it is convenient to factor out the term 
    $$ T_k:= \frac{q^k\qraising{q^{-n}}{q}{k}\qraising{abq^{n+3}}{q}{k}}{\qraising{aq}{q}{k+1}\qraising{q}{q}{k}},$$
    from each coefficient. Notice that with the notation \eqref{basicform}, for $k\geq 1$,
    $$\coef{k}{\lj{n}{a}{q^2b}{x}}=\frac{q^k\qraising{q^{-n}}{q}{k}\qraising{abq^{n+3}}{q}{k}}
    {\qraising{aq}{q}{k}\qraising{q}{q}{k}}
    =(1-aq^{k+1}) T_k.$$
    Also
    \begin{align*}
    \coef{k}{\tfrac{a(1-q^n)(1-abq^{n+3})}{q^{n-2}(1-aq)(1-aq^2)}x\lj{n-1}{q^2a}{q^2b}{x}}
    &=\tfrac{a(1-q^n)(1-abq^{n+3})}{q^{n-2}(1-aq)(1-aq^2)} \tfrac{q^{k-1}\qraising{q^{-n+1}}{q}{k-1}\qraising{abq^{n+4}}{q}{k-1}}
    {\qraising{aq^3}{q}{k-1}\qraising{q}{q}{k-1}} \\
    &=\frac{a(1-q^n)(1-q^{k})}{q^{n-1}(1-q^{-n}) } T_k\\
    &=-aq(1-q^{k}) T_k. 
    \end{align*}
    Finally,
     $$\coef{k}{\lj{n}{qa}{qb}{x}}=\frac{q^k\qraising{q^{-n}}{q}{k}\qraising{abq^{n+3}}{q}{k}}
    {\qraising{aq^2}{q}{k}\qraising{q}{q}{k}}
    =(1-aq) T_k.$$
    Then, \eqref{contiguousrelation1} follows from noticing that
    $$
    -(1-aq^{k+1})  =-aq(1-q^{k-1})   -(1-aq)  , \qquad \text{for } k\geq 1.
    $$

As for $k=0$ (the independent term), it is easy to see that on both sides it is equal to $1-aq$.

Relation \eqref{contiguousrelation2} follows from factoring out $$S_k:=\frac{q^k\qraising{q^{-n}}{q}{k}\qraising{abq^{n+1}}{q}{k+1}}{\qraising{q}{q}{k}\qraising{aq^2}{q}{k}}.$$ 
Then the equality of the $k$-th coefficients on both sides follows from the identity
$$(1-aq^{k+1})(1+abq^{2n+1}-abq^{n+1}-bq^{n})= $$
$$(1-bq^{n})(1-aq^{n+1})-abq^{n+1}(1-aq^{k+1})(1-q^{k})+aq^{k+1}(q^{n-k}-1)(1-abq^{n+k+1}).
$$

To prove \eqref{contiguousrelation3}, we can factor out  the term $U_k:=\frac{q^k\qraising{q^{-n}}{q}{k}\qraising{abq^{n}}{q}{k}}{\qraising{q}{q}{k}\qraising{aq}{q}{k}}$ to obtain that
\begin{align*}
    &\coef{k}{(1-abq^{2n})\lj{n}{a}{b/q}{qx}-(1-q^n)\lj{n-1}{a}{b}{x} }\\
    &=(1-abq^{2n}) U_k+(q^{-n+k}-1)U_k\\
    &=q^{n-k}(1-abq^{n+k}) U_k\\
    &=\coef{k}{q^n(1-abq^n)\lj{n}{a}{b}{x}}
    \end{align*}
\end{proof}

\begin{proof}[Proof of Theorem \ref{Thm:propertieslittleqjacobi2}]

We start by proving that $\lj{n}{a}{b}{x}$ has the lmesh smaller than $q$. By \eqref{lmeshRplus},  $\lj{n}{a}{b}{x}\in\P_n^q(\R_{>0})$  is equivalent to showing that 
$$\lj{n}{a}{b}{x}\prec \lj{n}{a}{b}{qx}.$$

Since the zeros of $\lj{n}{a}{b}{qx}$ are the zeros of $\lj{n}{a}{b}{x}$ scaled by $q^{-1}$, we readily get that  
\begin{equation}\label{lmeshauxiliar2}
    \lj{n}{a}{b}{x}\ll \lj{n}{a}{b}{qx}. 
\end{equation}
So, by Remark \ref{rem:interlacing.partial.order}, we only need to show that the polynomials $\lj{n}{a}{b}{x}$ and $\lj{n}{a}{b}{qx}$ have a common interlacer. We will prove that $\lj{n-1}{a}{qb}{x}$ is this interlacer. Indeed, \eqref{eq:interlacing1} directly yields that $\lj{n}{a}{b}{x}\prec\lj{n-1}{a}{qb}{x}$. On the other hand,  by the orthogonality of the little $q$-Jacobi polynomials we know that $\lj{n}{a}{qb}{x}\prec\lj{n-1}{a}{qb}{x}$.
Identity \eqref{contiguousrelation3} gives us the contiguous relation
\begin{equation}\label{contiguousrelation32}
   (1-abq^{2n+1})\lj{n}{a}{b}{qx}=(1-q^n)\lj{n-1}{a}{qb}{x}+ q^n(1-abq^{n+1})\lj{n}{a}{qb}{x},
\end{equation}
where all the polynomials are multiplied by a constant (that does not change sign). Thus, the conditions of Lemma \ref{Jordaaninterlacing} are satisfied and we conclude that  
\begin{equation}\label{lmeshauxiliar1}
    \lj{n}{a}{b}{qx}\prec \lj{n-1}{a}{qb}{x}.
\end{equation}

To prove statement \textit{(i)} of the theorem, consider the $q$-hypegeometric expression of $\lj{n}{a}{b}{x}$ from \eqref{def:LittleqJacobi}. Using \eqref{qderivative1} and  \eqref{corollary:qderivativeqhyper} we obtain that, up to a multiplicative constant, the polynomial $\lj{n-1}{qa}{qb}{x}$ is the $q$-derivative of $\lj{n}{a}{b}{x}$:
   \begin{align}
       \D_q(\lj{n}{a}{b}{x})&=\D_q\left(\pQq{2}{1}{q^{-n}, abq^{n+1}}{aq}{q}{qx}\right)\nonumber\\
       &=q\frac{(1-q^{-n})(1-abq^{n+1})}{(1-q)(1-aq)} \pQq{2}{1}{q^{-n+1}, abq^{n+2}}{aq^2}{q}{qx}\nonumber\\
       &=q\frac{(1-q^{-n})(1-abq^{n+1})}{(1-q)(1-aq)}\lj{n-1}{qa}{qb}{x}.\label{qderivativelittleqJacobi}
   \end{align}

As it was already established, $\lj{n}{a}{b}{x}\in\P_n^q(\R_{>0})$. By Proposition \ref{prop:lmeshinterlaing} we conclude that 
$$\lj{n}{a}{b}{x}\prec\lj{n-1}{qa}{qb}{x}.$$

For the statement \textit{(ii)} of the theorem we will use the partial order $\ll$ and Remark \ref{rem:interlacing.partial.order}.

Equation \eqref{eq:interlacing3} of Theorem \ref{thmx:interlacings} (with the change of variables $a\to qa$ and $b\to \tfrac{b}{q}$ for the second equation), yields:
 \begin{equation*}
     \lj{n}{a}{q^2b}{x}\prec\lj{n}{qa}{qb}{x} \qquad \text{for } 0<aq<1 \text{ and }bq<1, \text{ and}
 \end{equation*}
\begin{equation*}
     \lj{n}{qa}{qb}{x}\prec\lj{n}{q^2a}{b}{x} \qquad \text{for } 0<aq^2<1 \text{ and }b<1.
 \end{equation*}
By \eqref{eq:interlacing.implies.monotonicity}, these two equations imply for $0<aq<1$ and $b<1$,
 \begin{equation}\label{newresultequation2}
\lj{n}{a}{q^2b}{x}\ll\lj{n}{qa}{qb}{x}\ll \lj{n}{q^2a}{b}{x}. 
 \end{equation}
By Theorem \ref{Thm:propertieslittleqjacobi}, for $0<aq,bq<1$, we get, 
 \begin{equation}\label{newresultequation3}
   \lj{n}{a}{q^2b}{x}\ll \lj{n}{q^2a}{b}{x}. 
 \end{equation}
 Gathering together $\eqref{newresultequation2}$ and \eqref{newresultequation3} we obtain for $0<aq<1$ and $bq<1$,
 \begin{equation}\label{newresultequation4}
     \lj{n}{a}{q^2b}{x}\ll \lj{n}{q^2a}{b}{x}. 
 \end{equation}
On the other hand, Equation \eqref{eq:interlacing2} with a modification of the parameters gives us
\begin{equation}\label{newresultequation5}
 \lj{n}{q^2a}{b}{x}\prec\lj{n-1}{q^2a}{q^2b}{x} \qquad \text{for } aq^3<1 \text{ and }bq<1.   
\end{equation}
Combining statement \textit{(i)},  \eqref{contiguousrelation4}, and Lemma \ref{Jordaaninterlacing} we have that 
\begin{equation}\label{newresultequation6}
       \lj{n}{a}{q^2b}{x}\prec\lj{n-1}{q^2a}{q^2b}{x}.
   \end{equation}
Equations \eqref{newresultequation5} and \eqref{newresultequation6} mean that $\lj{n}{q^2a}{b}{x}$ and $\lj{n}{a}{q^2b}{x}$ have a common interlacing. Together with \eqref{newresultequation4} and Remark \ref{rem:interlacing.partial.order} we conclude that
\begin{equation*}
     \lj{n}{a}{q^2b}{x}\prec\lj{n}{q^2a}{b}{x}.
 \end{equation*}

For the statement \textit{(iii)} of the theorem we assume $b<0$, so that $b<q^2b$. By Theorem \ref{Thm:propertieslittleqjacobi}, we know that
$$\lj{n}{a}{b}{x}\ll \lj{n}{a}{q^2b}{x}.$$ 
On the other hand, these two polynomials have a common interlacer $\lj{n+1}{a}{b}{x}$. Indeed,
 \eqref{eq:interlacing2} asserts that
$\lj{n+1}{a}{b}{x}\prec\lj{n}{a}{q^2b}{x}$, 
 while the orthogonality of the little $q$-Jacobi implies $\lj{n+1}{a}{b}{x}\prec\lj{n}{a}{b}{x}$.
 Therefore, by Remark \ref{rem:interlacing.partial.order} we conclude $\lj{n}{a}{b}{x}\prec\lj{n}{a}{q^2b}{x}$.
\end{proof}

\subsection{Proof of the remaining results}

In this last section we will prove Corollary \ref{cor:interlacinglittleqJacobi}, Proposition~\ref{littleqjacobimultiplicity}, and Theorem~\ref{thm:lmeshqbessel}.

Corollary \ref{cor:interlacinglittleqJacobi} follows from Theorems \ref{Thm:propertieslittleqjacobi} and \ref{Thm:propertieslittleqjacobi2}.

\begin{proof}[Proof of Corollary \ref{cor:interlacinglittleqJacobi}] Consider $q^2\leq t_1,t_2\leq 1$. For part \textit{(i)} we have the assumption that $a,b>0$, so $q^2b\leq t_1b \leq b$ and $a\leq t_2a \leq q^2 a$. Thus, monotonicity of the zeros (Theorem \ref{Thm:propertieslittleqjacobi}) implies that
 $$
 \lj{n}{a}{q^2b}{x} \ll \lj{n}{a}{t_1b}{x}\ll \lj{n}{t_2a}{b}{x} \ll \lj{n}{q^2a}{b}{x}.
 $$
By Theorem \ref{Thm:propertieslittleqjacobi2} \textit{(i)}, $\lj{n}{a}{q^2b}{x}\prec\lj{n}{q^2a}{b}{x}$, and we conclude that
\begin{equation*}
     \lj{n}{a}{t_1b}{x}\prec\lj{n}{t_2a}{b}{x}.
 \end{equation*}
The proof of part \textit{(ii)} is similar. Now we have $b<0$, so inequalities are reversed $b\leq t_1b\leq q^2b$ and Theorem \ref{Thm:propertieslittleqjacobi2} yields
$\lj{n}{a}{b}{x}\ll \lj{n}{a}{t_1b}{x}\ll \lj{n}{a}{q^2b}{x}$. On the other hand, Theorem \ref{Thm:propertieslittleqjacobi} \textit{(iii)}, asserts that $\lj{n}{a}{b}{x}\prec\lj{n}{a}{q^2b}{x}$, and we conclude
that $ \lj{n}{a}{b}{x}\prec\lj{n}{a}{t_1b}{x}$.
\end{proof}

\medskip

Proposition~\ref{littleqjacobimultiplicity} follows by induction on the degree $n$ and the contiguous relations \eqref{contiguousrelation3} and \eqref{contiguousrelation4}.

\begin{proof}[Proof of Proposition~\ref{littleqjacobimultiplicity}]
  First, we will prove relation \eqref{factorization}. Notice that $\lj{0}{a}{q^n}{x}=1$ so the case $j=n$ is trivial: 
\begin{equation}\label{Theorem:littlejacobi0}
    \lj{n}{a}{q^{-n}}{x}=\pQq{2}{1}{q^{-n},aq}{aq}{q}{qx}=\pQq{1}{0}{q^{-n}}{\cdot}{q}{qx}=E_n(qx).
\end{equation}
The proof for the other cases follows by induction on the degree of $\lj{n}{a}{q^{j}}{x}$. The base case $n=1$ is already covered by \eqref{Theorem:littlejacobi0}. Now suppose that for $n\in\N$ we have
\begin{align}\label{Theorem:inductionhypothesis}
    \lj{n}{a}{q^{-j}}{x}=E_{j}(qx)\lj{n-j}{a}{q^j}{q^{-j}x}, \quad \text{for } j\in\{1,\cdots,n\}.
\end{align}
We now prove the claim for degree $n+1$ and $j\in\{1,\cdots,n\}$  (since the case $j=n+1$ is already covered in \eqref{Theorem:littlejacobi0}). We first use the contiguous relation \eqref{contiguousrelation4} to obtain
\begin{align*}
    q^{n-j+1} &\left(1-aq^{n+1} \right) \lj{n+1}{a}{q^{-j}}{x}\\
    &=(1 - a q^{2n-j+2}) \left(1 - q^{-j+1}x\right) \lj{n}{a}{q^{-(j-1)}}{x}-\left(1- q^{n-j+1}\right) \lj{n}{a}{q^{-j}}{x}. \nonumber
\end{align*}
By induction hypothesis \eqref{Theorem:inductionhypothesis}, we can rewrite the right hand side as 
\begin{align}\label{Theorem:littlejacobimultiplicity3}
   & q^{n-j+1} \left(1-aq^{n+1} \right) \lj{n+1}{a}{q^{-j}}{x}=\\
    &=E_j(qx)\left((1 - a q^{2n-j+2})\lj{n-j+1}{a}{q^{j-1}}{\frac{x}{q^{j-1}}}
    -\left(1- q^{n-j+1}\right)\lj{n-j}{a}{q^j}{\frac{x}{q^j}}\right). \nonumber
\end{align}
On the other hand, the contiguous relation \eqref{contiguousrelation3}, with degree $n-j+1$, parameter $b=q^{-j}$, and variable $\frac{x}{q^j}$, yields
\begin{align}\label{Theorem:littlejacobimultiplicity4}
    q^{n-j+1} \left(1-aq^{n+1} \right)& \lj{n-j+1}{a}{q^{j}}{\frac{x}{q^j}}=\\&(1 - a q^{2n-j+2})\lj{n-j+1}{a}{q^{j-1}}{\frac{x}{q^{j-1}}}
    -\left(1- q^{n-j+1}\right)\lj{n-j}{a}{q^j}{\frac{x}{q^j}}.\nonumber
\end{align}
Gathering \eqref{Theorem:littlejacobimultiplicity3} and \eqref{Theorem:littlejacobimultiplicity4}, we conclude that
\begin{equation*}
        \lj{n+1}{a}{q^{-j}}{x}=E_{j}(qx)\lj{n-j+1}{a}{q^j}{q^{-j}x} \qquad \text{for }j=1,\dots,n.
\end{equation*}

We finally prove that $\lj{n}{a}{q^{-k}}{x}\in\overline{\P_n^q}((0,1))$ using the previous relation.

Recall that by the orthogonality that $\lj{n-j}{a}{q^j}{x}\in \P_n^q((0,1))$. Then, after the change of variable $x\mapsto xq^{-j}$ we obtain that $$\lj{n-j}{a}{q^j}{q^{-j}x}\in\P_n^q((0,q^j)).$$ Finally, since the zeros of $E_j(qx)$ are $q^{j-1}<q^{j-2}<\cdots<q<1,$ we conclude that
$$\lj{n}{a}{q^{-j}}{x}=E_{j}(qx)\,\lj{n-j}{a}{q^j}{q^{-j}x}\in \overline{\P^q_n}((0,1)).$$
\end{proof}
\begin{proof}[Proof of Proposition \ref{littleqjacobimultiplicity0}]
     For $k\in\{1,\dots,n\}$, in \eqref{normalizedqJacobi}, the $j$-th coefficient of the polynomial in the right-hand side vanishes when $j<k-1$, so that \eqref{normalizedqJacobi} can be written as 
     \begin{equation}
        \qraising{q^{-k+1}}{q}{n}\lj{n}{q^{-k}}{b}{x}=\sum_{j=k}^n\frac{\qraising{q^{-n}}{q}{j}}{\qraising{q}{q}{j}}\qraising{bq^{n-k+1}}{q}{j}\qraising{q^{j-k+1}}{q}{n-j}(qx)^j.
\end{equation}
Rearranging the summation we get 
\begin{equation}\label{normalizedqJacobi1}
        \qraising{q^{-k+1}}{q}{n}\lj{n}{q^{-k}}{b}{x}=(qx)^k\sum_{j=0}^{n-k}\frac{\qraising{q^{-n}}{q}{j+k}}{\qraising{q}{q}{j+k}}\qraising{bq^{n-k+1}}{q}{j+k}\qraising{q^{j+1}}{q}{n-j-k}(qx)^{j}.
\end{equation}
Direct computations using the definition of the $q$-Pochhammer symbol show that, for the $j$-th polynomial coefficient of the summation on the right-hand side of \eqref{normalizedqJacobi1} we have
\begin{multline*}
    \frac{\qraising{q^{-n}}{q}{j+k}}{\qraising{q}{q}{j+k}}\qraising{bq^{n-k+1}}{q}{j+k}\qraising{q^{j+1}}{q}{n-j-k}q^{j}=\\
    (-1)^kq^{\binom{k}{2}-nk}\qraising{bq^{n-k+1}}{q}{k}\qraising{q^{k+1}}{q}{n-k}\frac{\qraising{q^{-n+k}}{q}{j}}{\qraising{q}{q}{j}}\frac{\qraising{bq^{n+1}}{q}{j}}{\qraising{q^{k+1}}{q}{j}}q^j.
\end{multline*}
Substituting this into \eqref{normalizedqJacobi1} gives 
$$
\qraising{q^{-k+1}}{q}{n}\lj{n}{q^{-k}}{b}{x}=(-1)^kq^{\binom{k}{2}-nk}\qraising{bq^{n-k+1}}{q}{k}\qraising{q^{k+1}}{q}{n-k}(qx)^k\lj{n-k}{q^{k}}{b}{x},
$$
which proves the desired result.
\end{proof}
\medskip
Theorem~\ref{thm:lmeshqbessel} follows from approximating the $q$-Bessel polynomials with little $q$-Jacobi polynomials, see \eqref{def:qbessel}, and using what has been already proved for the little $q$-Jacobi polynomials.

\begin{proof}[Proof of Theorem~\ref{thm:lmeshqbessel}]
Recall from \eqref{def:qbessel} that the $q$-Bessel polynomials $\besselq{n}{b}{x}$ can be obtained as the limit $\lj{n}{a}{\frac{b}{qa}}{x}$ when $a\to0$. Then, interlacing \eqref{eq:Besselinterl} is a direct consequence of Corollary \ref{cor:interlacinglittleqJacobi} and Lemma \ref{lemma:limit} (\ref{limitinterlacing}).

On the other hand, Lemma \ref{lemma:limit} (\ref{limitlmesh}) implies that $\besselq{n}{b}{x}\in\overline{\P^q_n}((0,1))$. In other words, the logaritmic mesh of the $q$-Bessel polynomials satisfies
$$\lmesh{\left[\besselq{n}{b}{x}\right]} \leq q.$$ 
To prove \eqref{eq:Besselrealzeros}, we just need to show that this inequity is strict. 
Assume that this is false; that is, that $\lmesh{\left[\besselq{n}{b}{x}\right]}=q$. This implies that there exist two zeros, $y_0\neq 0$ and $y_1\neq 0$, of $\besselq{n}{b}{x}$ such that $y_1=qy_0$. From \cite[(14.22.3)]{koekoek2010hypergeometric}, we have the $q$-difference equation,
\begin{equation}\label{eq:qbessellemsh4}
    (cx+1)\, B\left(x\right)=b\, x\, B\left(qx\right)+(1-x)\, B\left(q^{-1}x\right),
\end{equation}
where $B(x)=\besselq{n}{b}{x}$ to simplify notation, and $c:=-q^{-n}(1-q^n)(1+bq^n)+b-1$ is a constant.

Evaluating $x=y_1=qy_0$ in \eqref{eq:qbessellemsh4} we obtain that 
$$by_1\, B\left(qy_1\right)=0.$$ 
Thus, $y_2:=q^2y_0$ is also a zero of $B$. Similarly, evaluating $x=q^2y_0$ in \eqref{eq:qbessellemsh4} we obtain that $q^3y_0$ is also a zero of $B$. Recursively, we can prove that $q^ky_0$ is a zero of $B$ for all $k\in\N$, contradicting that $B$ is a polynomial of degree $n$. This shows that $\besselq{n}{b}{x}\in\P_n^q((0,1))$, as claimed.
\end{proof}

 \section*{Acknowledgments}
The first author was partially supported by Simons Foundation Collaboration Grants for Mathematicians (grant MPS-TSM-00710499)).
He also acknowledges the support of the project PID2021-124472NB-I00, funded by MCIN/AEI/10.13039/501100011033 and by ``ERDF A way of making Europe'', as well as the support of Junta de Andaluc\'{\i}a (research group FQM-229 and Instituto Interuniversitario Carlos I de F\'{\i}sica Te\'orica y Computacional). 

The third author was partially supported by AMS-Simons travel grant. He expresses his gratitude for the warm hospitality and stimulating atmosphere at Baylor University.

\end{document}